\documentclass[svgnames,12pt]{amsart}
\usepackage[utf8]{inputenc}
\usepackage[english]{babel}
\usepackage{enumerate}
\usepackage{latexsym}
\usepackage[pdftex]{graphicx}
\usepackage{amssymb, amsmath}
\usepackage{srcltx}
\usepackage{xcolor}
\usepackage{hyperref}

\allowdisplaybreaks

\newtheorem{theorem}{Theorem}
\newtheorem{lemma}{Lemma}
\newtheorem{proposition}{Proposition}
\newtheorem{remark}{Remark}
\newtheorem{example}{Example}

\newtheorem{definition}{Definition}
\newtheorem{corollary}{Corollary}

\newtheorem{problem}{Problem}
\newtheorem{conjecture}{Conjecture}

\def\R{{\mathbb R}}

\newcommand{\ce}{\mathcal E}

\newcommand{\beq}{\begin{equation}}
\newcommand{\eeq}{\end{equation}}
\newcommand{\beqna}{\begin{eqnarray*}}
\newcommand{\eeqna}{\end{eqnarray*}}
\newcommand{\beqn}{\begin{equation*}}
\newcommand{\eeqn}{\end{equation*}}
\newcommand{\bprop}{\begin{proposition}}
\newcommand{\eprop}{\end{proposition}}
\newcommand{\bt}{\begin{theorem}}
\newcommand{\et}{\end{theorem}}
\newcommand{\bex}{\begin{example}}
\newcommand{\eex}{\end{example}}
\newcommand{\bc}{\begin{corollary}}
\newcommand{\ec}{\end{corollary}}
\newcommand{\bl}{\begin{lemma}}
\newcommand{\el}{\end{lemma}}
\newcommand{\bprob}{\begin{problem}}
\newcommand{\eprob}{\end{problem}}
\newcommand{\br}{\begin{remark}}
\newcommand{\er}{\end{remark}}
\newcommand{\bd}{\begin{definition}}
\newcommand{\ed}{\end{definition}}

\begin{document}

\title[Unique determination of  ellipsoids by their dual volumes]
{Unique determination of  ellipsoids by their dual volumes and the moment problem}

\author{Sergii Myroshnychenko, Kateryna Tatarko, Vladyslav Yaskin}\thanks{The   authors were supported in part by NSERC. The first author was supported in part by PIMS postdoctoral fellowship.}

\address{Department of Mathematical and Statistical Sciences, University of Alberta, Edmonton, Alberta T6G 2G1, Canada} \email{myroshny@ualberta.ca}
  \email{tatarko@ualberta.ca}  \email{yaskin@ualberta.ca}



\maketitle

\begin{abstract}
	Gusakova and Zaporozhets conjectured that ellipsoids in $\mathbb R^n$ are unique\-ly determined  (up to an isometry) by their Steiner polynomials.  Petrov and Tarasov  confirmed this conjecture in $\mathbb R^3$. In this paper we solve the dual problem. 	
	We show that any ellipsoid in $\R^n$ centered at the origin is uniquely determined (up to an isometry)  by its dual Steiner polynomial. To prove this result we reduce it to a problem of moments. As a by-product we give an alternative proof of the result of Petrov and Tarasov.
\end{abstract}

\section{Introduction}

The study of behavior of volume under the Minkowski (vector) addition is the main focus of the Brunn-Minkowski theory. 
Let $K$ be a convex body in $\mathbb R^n$. Denote by $B_2^n$ the unit Euclidean ball in $\mathbb R^n$. The classical Steiner formula asserts that for every $\epsilon > 0$,
 $$\mathrm{vol}(K+\epsilon B_2^n) =  \sum_{i=0}^n\kappa_{n-i} V_i(K) \epsilon^{n-i},$$
where $\mathrm{vol}$ is the Lebesgue measure on $\mathbb R^n$, the addition $+$ is the Minkowski addition, and $\kappa_{n-i}$ is the volume of $B_2^{n-i}$. The coefficients   $V_i(K)$ are known as the intrinsic volumes of $K$. The geometric interpretation of some of these quantities is the following: $V_n(K)$ is the volume of $K$, $V_{n-1}(K)$ is (a multiple of) the surface area, $V_1(K)$ is (a multiple of) the mean width. A thorough discussion of intrinsic volumes can be found in \cite{Sch}.

  Gusakova and Zaporozhets   asked if ellipsoids are uniquely determined by their intrinsic volumes (up to an isometry). Namely, they conjectured the following. 
 \begin{conjecture} \label{conj:main1}
 Let $\ce_1$ and $\ce_2$	be two ellipsoids in $\R^n$ such that $  V_1(\ce_1)= V_1(\ce_2)$, $  V_2(\ce_1)= V_2(\ce_2)$,..., $  V_n(\ce_1)= V_n(\ce_2)$. Then $\ce_1$ and $\ce_2$ are congruent.
\end{conjecture}

 Petrov and Tarasov \cite{PT} confirmed this conjecture in $\mathbb R^3$. For higher dimensions, the problem is still open. 

The purpose of this paper is to show that a similar question for dual volumes has a positive answer in all dimensions. Dual volumes were introduced by Lutwak \cite{L} within the framework of the  dual Brunn-Minkowski theory. In this theory the Minkowski addition of convex bodies is replaced by the radial addition of star bodies (see the next section for definitions of these concepts). The dual version of the Steiner formula  asserts that
$$\mathrm{vol}(K\,\tilde{+} \,\epsilon B_2^n) =   \sum_{i=0}^n{ n\choose i} \widetilde V_i(K) \epsilon^{n-i},$$
where $K$ is a star body in $\mathbb R^n$  and $\tilde{+}$ is the radial addition.
The coefficients   $\widetilde V_i(K)$ are called the  dual   volumes. Note that $\widetilde V_n(K)$ is equal to the volume of $K$.
Denoting by $\rho_K$ the radial function of $K$, one can write the dual volumes of $K$ as follows:
\begin{equation}\label{tildeVi} 
\widetilde V_i (K) = \frac{1}{n} \int_{S^{n-1}} \rho_K^i(\theta) \, d\theta,
\end{equation} 
where the integration is with respect to the spherical Lebesgue measure.

 Note that   while the intrinsic volumes are invariant under translations, the dual volumes depend on the choice of the origin. Both the intrinsic volumes and dual volumes are invariant under orthogonal transformations. 

Our main result is the following. 

\begin{theorem} \label{thm:main1}
	Let $\ce_1$ and $\ce_2$	be two ellipsoids in $\R^n$, $n\ge 2$, centered at the origin such that $\widetilde V_1(\ce_1)=\widetilde V_1(\ce_2)$, $\widetilde  V_2(\ce_1)= \widetilde V_2(\ce_2)$,..., $\widetilde  V_n(\ce_1)=\widetilde V_n(\ce_2)$. Then $\ce_1$ and $\ce_2$ are congruent.
\end{theorem}

As one can see, the right-hand side of formula (\ref{tildeVi}) makes sense for all real $i$.  This allows to use  (\ref{tildeVi}) as a definition of dual volumes of any order $i$. 
In view of this remark, in the statement of Theorem \ref{thm:main1} the collection of the dual volumes $\{ \widetilde V_{i}  \}_{i=1}^n$can be replaced by any $n$-tuple  of the form $\{ \widetilde V_{i_k} \}_{k=1}^n$, where $i_1$,..., $i_n$ are distinct non-zero real numbers from the interval $(-2,n]$. In some cases one can take numbers from a larger interval, as will be discussed later. 

To prove Theorem \ref{thm:main1} we reduce it to a   problem of moments. Using this idea we also give an alternative proof of the result of Petrov and Tarasov for the intrinsic volumes of ellipsoids in $\mathbb R^3$. In the final section we investigate Conjecture \ref{conj:main1} in the setting of ellipsoids of revolution.

\section{Preliminaries}
Let $K$ be a convex body in $\mathbb R^n$, that is a convex compact set with non-empty interior. The support function of $K$ is defined by 
$$h_K(x) = \max_{y\in K} \langle x,y\rangle,\qquad x\in \mathbb R^n. $$
The first intrinsic volume of $K$ is a multiple of its mean width:
\begin{equation}\label{1stIV}
V_1(K) = \frac{1}{\kappa_{n-1}} \int_{S^{n-1}} h_K(\theta) d\theta,
\end{equation}
where the integration is with respect to the spherical Lebesgue measure.

Let  $\ce$ be an ellipsoid in $\mathbb R^n$ with semi-axes $a_1, \dots, a_n$. Its volume is given by
$$\mathrm{vol}(\ce) = \kappa_n a_1  \cdots  a_n.$$  Below we will always assume that $\ce$ is centered at the origin and its axes coincide with the coordinate axes; i.e.,
\begin{equation}\label{ce}
\ce=\left\{x\in \mathbb R^n: \frac{x_1^2}{a_1^2} + \cdots + \frac{x_n^2}{a_n^2}\le 1\right\}.
\end{equation} 
The support function of this ellipsoid $\ce$ is given by
$$h_{\ce} (x) = \left({a_1^2} {x_1^2} + \cdots + {a_n^2}  {x_n^2}\right)^{1/2}, \qquad x\in \mathbb R^n.$$

For a convex body $K\subset \mathbb R^n$ its polar body $K^\circ$ is defined as follows:
$$K^\circ = \{ x\in \mathbb R^n: \langle x,y\rangle \le 1, \forall y \in K\}.$$
For the ellipsoid $\ce$ defined above by (\ref{ce}), its polar body is also an ellipsoid given by 
$$\ce^\circ =\left\{x\in \mathbb R^n: {a_1^2} {x_1^2} + \cdots + {a_n^2}  {x_n^2}\le 1\right\}.$$

We say that   a compact set   $K\subset \mathbb R^n$   is star-shaped about the origin $o$ if for every point $x\in K$ each point of the interval $[o,x)$ is an interior point of $K$. 
The {\it Minkowski functional} of $K$ is defined by
\begin{align*}
\|x\|_K=\min \{a\ge 0: x \in aK \}, \qquad x\in \mathbb R^n.
\end{align*}
We say that $K$ is a {\it star body} if it is compact, star-shaped about the origin and its Minkowski functional  is a  continuous function on $\mathbb R^n$.

The Minkowski functional of   ellipsoid (\ref{ce}) is given by 
$$\|x\|_{\ce} = \left(\frac{x_1^2}{a_1^2} + \cdots + \frac{x_n^2}{a_n^2}\right)^{1/2}, \qquad x \in \mathbb R^n.$$
For the Euclidean ball $B_2^n$, we will denote its Minkowski functional simply by $|x|$. 

If $K$ is a convex body containing the origin in its interior then we have the following relation:
$$\|x\|_K = h_{K^\circ}(x), \qquad x\in \mathbb R^n.$$

The {\it radial function} of a star body  $K$ is defined by \begin{align*}
\rho_K(\xi) = \max\{a \ge 0 :a \xi \in K\}, \quad \xi\in S^{n-1}.
\end{align*}
 Clearly,   $\rho_K(\xi) =  \|\xi\|_K^{-1}$ for $\xi\in S^{n-1}$.

If $K$ and $L$ are star bodies, and $\alpha$ and $\beta$ are positive numbers, then the radial sum $\alpha K\,\tilde{+}\, \beta L$ is  the star body with the radial function
$$\rho_{\alpha K\tilde{+} \beta L} = \alpha \rho_{K}+\beta \rho_{L}.$$

Let us now recall some basic facts about the Gamma-function and fractional derivatives. If    $z\in \mathbb C$ with $\Re(z)>0$, then 
$$\Gamma (z) = \int_0^\infty t^{z-1} e^{-t}\, dt.$$
The Gamma-function can be extended to the set $ \mathbb C\setminus (-\mathbb N\cup\{0\})$ using the relation $$\Gamma (z+1)=z\Gamma (z).$$

Let $f$ be a continuous integrable function on $\mathbb [0,\infty)$ which is infinitely smooth in some neighborhood of zero.   If $q\in \mathbb C$, $-1<\Re(q)<0$, the fractional derivative of $f$ of order $q$ at zero is defined by
$$f^{(q)}(0) = \frac{1}{\Gamma(-q)}\int_0^\infty t^{-1-q} f(t)\,dt.$$
This formula is in fact valid for all  $q\in \mathbb C$, $\Re(q)<0$, if $ t^{-1-q} f(t)$ decays at infinity sufficiently fast.

For $q$ with  $\Re(q)>0$ that are not positive integers, the fractional derivatives are defined using analytic continuation;  see, e.g., \cite{GS} or \cite{K}. 
In particular, if $m$ is a positive integer, then for $m-1<\Re(q)<m$ we have 
$$f^{(q)}(0) = \frac{1}{\Gamma(-q)}\int_0^\infty t^{-1-q} \left(f(t)- \sum_{k=0}^{m-1} \frac{f^{(k)}(0)}{k!}t^k\right)\,dt.$$
Fractional derivatives of positive integer orders coincide with ordinary derivatives (up to a sign). 

The reader is referred to the books \cite{G}, \cite{K}, \cite{Sch} for other facts on the topics discussed above. 

\section{Main results}\label{SectionMR}

It is known that the intrinsic volumes of an ellipsoid and its polar are related by the following formula. 
\begin{equation}\label{KZ} V_i(\ce) = \frac{\kappa_{i}}{\kappa_n\kappa_{n-i}}V_n(\ce) V_{n-i}(\ce^\circ);
\end{equation} 
see \cite[Prop. 4.8]{KZ}.

Let us prove a similar formula for the dual volumes. 

\begin{lemma}
	Let $\ce$ be a centered ellipsoid in $\mathbb R^n$. Then for all real $i$ we have 
	 \begin{equation}\label{relation_dual}
	\widetilde{V}_i(\ce) = \frac{1}{\kappa_n} \widetilde{V}_n(\ce)\widetilde{V}_{n-i}(\ce^\circ).
	\end{equation}
\end{lemma}
\begin{proof}
 The formula is trivial when $i=n$, so we will assume that $i\ne n$.
 We have
 	\begin{align*} 
 \widetilde{V}_i(\ce) &= \frac{1}{n}\int_{S^{n-1}} \|\theta\|_{\ce}^{-i} \, d\theta  = \frac{n-i}{n(2^{n-i}-1)}\int_{S^{n-1}} \int_1^2 r^{n-i-1} \, dr \, \|\theta\|_{\ce}^{-i} \, d\theta\\
 & = \frac{n-i}{n(2^{n-i}-1)}\int_{2 B_2^{n}\setminus B_2^n}  \|x\|_{\ce}^{-i} \, dx.
 \end{align*}
 Assuming without loss of generality that $$\|x\|_{\ce} = \left(\frac{x_1^2}{a_1^2} + \cdots + \frac{x_n^2}{a_n^2}\right)^{1/2}, \qquad x \in \mathbb R^n,$$
 and making a  change of variables, we get
	\begin{align*}   \widetilde{V}_i(\ce) &=   \frac{n-i}{n(2^{n-i}-1)}\int_{2 B_2^{n}\setminus B_2^n}  \left(\frac{x_1^2}{a_1^2} + \cdots + \frac{x_n^2}{a_n^2}\right)^{-i/2} \, dx\\
	&= a_1\dots a_n \frac{n-i}{n(2^{n-i}-1)}\int_{2 \ce^\circ\setminus \ce^\circ}  \left( {x_1^2}  + \cdots +  {x_n^2} \right)^{-i/2} \, dx\\
		&= a_1\dots a_n \frac{n-i}{n(2^{n-i}-1)} \int_{S^{n-1}}
	 \int_{\|\theta\|_{\ce^\circ}^{-1}}^{2\|\theta\|_{\ce^\circ}^{-1}} r^{n-i-1} \, dr \, d\theta\\
	 	&= a_1\dots a_n  \frac{1}{n}\int_{S^{n-1}}
	 \|\theta\|_{\ce^\circ}^{-n+i} \, d\theta\\
	 &=\frac{1}{\kappa_n} \widetilde{V}_n(\ce)\widetilde{V}_{n-i}(\ce^\circ).
  \end{align*}
\end{proof}

We will now obtain a representation of dual   volumes as certain moments.

\begin{lemma}\label{MainFormula}
	Let $\ce$ be an ellipsoid centered at the origin with semi-axes $a_1$,..., $a_n$. If $i$ is a  real number  (not necessarily an integer) such that $0<i<n$, then 
	\begin{equation}  \label{FormulaW_i1}
	\widetilde{V}_i(\ce) = \frac{4 \pi^{n/2}}{n \Gamma\left(\frac{n-i}2\right)\Gamma\left(\frac{i}{2}\right) }   \int\limits_0^\infty \frac{u^{i-1}}{\sqrt{(1 + \frac{u^2}{a_1^2})\cdots (1 + \frac{u^2}{a_n^2})}}\, du.
	\end{equation}
	
	If $-2<i<0$, then
	
		\begin{equation}  \label{FormulaW_i2}
	\widetilde{V}_i(\ce) = \frac{4 \pi^{n/2}}{n \Gamma\left(\frac{n-i}2\right)\Gamma\left(\frac{i}{2}\right) }   \int\limits_0^\infty u^{i-1}\left( \frac{1}{\sqrt{(1 + \frac{u^2}{a_1^2})\cdots (1 + \frac{u^2}{a_n^2})}}-1\right) \, du.
	\end{equation}
	
	If $n<i<n+2$, then
	\begin{equation}  \label{FormulaW_i3}	\widetilde{V}_i(\ce) = \frac{4 \pi^{n/2}}{n \Gamma\left(\frac{n-i}2\right)\Gamma\left(\frac{i}{2}\right) }   \int\limits_0^\infty u^{i-1} \left(\frac{1}{\sqrt{(1 + \frac{u^2}{a_1^2})\cdots (1 + \frac{u^2}{a_n^2})}}-\frac{a_1\cdots a_n}{u^n}\right) \, du  .
	\end{equation}
	
\end{lemma}
\begin{proof} Applying an orthogonal transformation if needed, we can assume that the Minkowski functional of $\ce$ is given by
	$$\|x\|_{\ce} = \left(\frac{x_1^2}{a_1^2} + \cdots + \frac{x_n^2}{a_n^2}\right)^{1/2}, \qquad x \in \mathbb R^n.$$
Recall that for all $i$ we have 
		\begin{equation*}\label{int}
		\widetilde{V}_i(\ce) = \frac{1}{n}\int_{S^{n-1}} \|\theta\|_{\ce}^{-i} \, d\theta.
		\end{equation*}
		
			First we will consider the case $i\in (0,n)$. Observe that
	\begin{equation}\label{repr}
	\int_{S^{n-1}} \|\theta\|_{\ce}^{-i} d\theta = \frac{2}{ \Gamma\left(\frac{n-i}2\right) } \int_{\R^n} \|x\|_{\ce}^{-i} e^{-|x|^2} \, dx.
	\end{equation}
	To check this, pass to the polar coordinates in the latter integral:
	$$ \int_{\R^n}  \|x\|_{\ce}^{-i} e^{-|x|^2} \, dx =  \int_{S^{n-1}}  \|\theta\|_{\ce}^{-i} \, d\theta \int_0^\infty r^{n-i-1} e^{-r^2} \, dr = \frac12 \Gamma\left(\frac{n-i}2\right)\int_{S^{n-1}} \|\theta\|_{\ce}^{-i}\, d\theta.$$
	
	Thus, we have 
	$$\widetilde{V}_i(\ce) = \frac{2}{n \Gamma\left(\frac{n-i}2\right) } \int_{\R^n}   \|x\|_{\ce}^{-i}   e^{-|x|^2} \, dx.
	$$

Now observe that, for $i>0$, we have
$$\|x\|_{\ce}^{-i}  = \frac{2}{\Gamma\left(\frac{i}{2}\right)} 	\int\limits_0^\infty s^{i-1}  e^{- \|x\|_{\ce}^{2}\, s^2}\, ds.$$
This can be checked by a   change of the variable in the latter integral.
	
	Therefore, 
	\begin{align*}
	\widetilde{V}_i(\ce)& = \frac{4}{n \Gamma\left(\frac{n-i}2\right)\Gamma\left(\frac{i}{2}\right) } 	\int\limits_0^\infty  s^{i-1} \int_{\R^n}  e^{-\left(\frac{x_1^2}{a_1^2} + \cdots + \frac{x_n^2}{a_n^2}\right)\, s^2}   e^{-|x|^2} \, dx\, ds \\
	& = \frac{4}{n \Gamma\left(\frac{n-i}2\right)\Gamma\left(\frac{i}{2}\right) }  	\int\limits_0^\infty s^{i-1} \int\limits_{\R}  e^{-x_1^2 \left(1 + \frac{s^{2}}{a_1^2} \right) }  dx_1 \cdots\int\limits_{\R} e^{-x_n^2 \left(1 + \frac{s^{2}}{a_n^2 } \right)}  dx_n \,  ds \\
		& = \frac{4}{n \Gamma\left(\frac{n-i}2\right)\Gamma\left(\frac{i}{2}\right) }  	\int\limits_0^\infty s^{i-1} \frac{\sqrt{\pi}}{\sqrt{1+\frac{s^2}{a_1^2} }}    \cdots \frac{\sqrt{\pi}}{\sqrt{1 + \frac{s^2}{a_n^2 }}} \, ds ,
    \end{align*} 
    as claimed. 
    
   To treat the case $i<0$, we will use the analytic continuation technique.
   Using the formula for $\widetilde{V}_i(\ce)$, $0<i<n$, proved above and making a change of variables, we get
   $$	\widetilde{V}_i(\ce) = \frac{2 \pi^{n/2}}{n \Gamma\left(\frac{n-i}2\right)\Gamma\left(\frac{i}{2}\right) }   \int\limits_0^\infty \frac{u^{i/2-1}}{\sqrt{(1 + \frac{u }{a_1^2})\cdots (1 + \frac{u }{a_n^2})}}\, du.$$
    Denoting
    $$f_{\ce}(u) =  \frac{1}{\sqrt{(1 + \frac{u}{a_1^2})\cdots (1 + \frac{u}{a_n^2})}},$$
    we see that 
  \begin{equation}\label{frdir}\widetilde{V}_i(\ce) = \frac{2 \pi^{n/2}}{n \Gamma\left(\frac{n-i}2\right)  }  f_{\ce}^{(-i/2)}(0), 
  \end{equation}
    where $f_{\ce}^{(-i/2)}(0)$ is the fractional derivative of $f$ of order $-i/2$ at zero.
    
    The left hand-side of \eqref{frdir} is an analytic function of $i\in \mathbb C$ and the right-hand side is analytic for in $\{i\in \mathbb C : \Re (i) < n\}$. Since they coincide on the interval $i\in (0,n)$, they coincide for all $i<n$. Thus formula \eqref{frdir} is valid for all $i<n$.
    In particular, for $-2<i<0$, we have 
    $$	\widetilde{V}_i(\ce) = \frac{2 \pi^{n/2}}{n \Gamma\left(\frac{n-i}2\right)\Gamma\left(\frac{i}{2}\right) }   \int\limits_0^\infty u^{i/2-1} \left( \frac{1}{\sqrt{(1 + \frac{u }{a_1^2})\cdots (1 + \frac{u }{a_n^2})}}-1\right) \, du.$$
    Replacing $u$ by $u^2$, we get the formula given in the statement of the lemma.

 To deal with the case when $i>n$, 
 we will first use \eqref{relation_dual} and then apply  \eqref{frdir} with  $n-i$ instead of $i$ and with $\ce^\circ$ instead of $\ce$:
 
 $$\widetilde{V}_i(\ce) = a_1\cdots a_n \widetilde{V}_{n-i}(\ce^\circ) = a_1\cdots a_n \frac{2 \pi^{n/2}}{n \Gamma\left(\frac{i}2\right)  }  f_{\ce^\circ}^{((i-n)/2)}(0), $$
 where 
 $$f_{\ce^\circ}(u) =  \frac{1}{\sqrt{(1 +  {u}{a_1^2})\cdots (1 +  {u}{a_n^2})}}.$$

    In particular, for $n<i<n+2$, we have 
  \begin{align*}	\widetilde{V}_i(\ce) &= \frac{2 \pi^{n/2}a_1\cdots a_n}{n \Gamma\left(\frac{n-i}2\right)\Gamma\left(\frac{i}{2}\right) }   \int\limits_0^\infty u^{(n-i)/2-1} \left( \frac{1}{\sqrt{(1 +  {u }{a_1^2})\cdots (1 +  {u }{a_n^2})}}-1\right) \, du\\
   &= \frac{4 \pi^{n/2}}{n \Gamma\left(\frac{n-i}2\right)\Gamma\left(\frac{i}{2}\right) }   \int\limits_0^\infty u^{i-1} \left(\frac{1}{\sqrt{(1 + \frac{u^2}{a_1^2})\cdots (1 + \frac{u^2}{a_n^2})}}-\frac{a_1\cdots a_n}{u^n}\right) \, du  .
  \end{align*}

\end{proof}

Now we are ready to prove our main result. 

{\bf Proof of Theorem \ref{thm:main1}.} Let $\ce_1$ be an ellipsoid with semi-axes $a_1, \ldots, a_n$, and $\ce_2$   an ellipsoid with semi-axes $b_1, \ldots, b_n$. Assume that 
$\widetilde V_i(\ce_1)=\widetilde V_i(\ce_2)$, for all $i=1,\ldots,n$.

Note that  $\widetilde  V_n(\ce_1)=\widetilde V_n(\ce_2)$ is the equality  of the volumes of  $\ce_1$ and $\ce_2$, which  gives $a_1\cdots a_n= b_1\cdots b_n$.

By Lemma \ref{MainFormula}, from the equalities of the dual volumes of orders from $1, \dots, n-1$ we get
$$	  \int\limits_0^\infty  \frac{u^{i-1}}{\sqrt{(1 + \frac{u^2}{a_1^2})\cdots (1 + \frac{u^2}{a_n^2})}}\, du =   \int\limits_0^\infty  \frac{u^{i-1}}{\sqrt{(1 +   \frac{u^2}{b_1^2})\cdots (1 + \frac{u^2}{b_n^2})}}  \, du ,$$
for all  $i=1,\ldots,n-1$. 

Denoting for brevity 
\begin{equation}\label{F}
F(u) =   \frac{1}{\sqrt{(1 + \frac{u^2}{a_1^2})\cdots (1 + \frac{u^2}{a_n^2})}} - \frac{1}{\sqrt{(1 +   \frac{u^2}{b_1^2})\cdots (1 + \frac{u^2}{b_n^2})}},
\end{equation}
we get 
$$	  \int\limits_0^\infty  {u^{i-1}} F(u)\, du =0,$$
for all  $i=1,\ldots,n-1$. 
This means that for any polynomial $P$ of degree at most $n-2$ we have 
$$	  \int\limits_0^\infty P(u) F(u) \, du=0.$$

Now observe that the function $F$ 
has at most $n-2$ positive real roots, unless it is identically equal to zero. Indeed, solving $F(u)=0$ is equivalent to solving
 $$ \left(1 + \frac{u^2}{a_1^2}\right)\cdots \left(1 + \frac{u^2}{a_n^2}\right)=\left(1 +   \frac{u^2}{b_1^2}\right)\cdots \left(1 + \frac{u^2}{b_n^2}\right) ,$$
 which can be written in the form
 $$
 1+u^2 q(u^2) +\frac{u^{2n}}{a_1^2\cdots a_n^2} = 1 + u^2 r(u^2) + \frac{u^{2n}}{b_1^2\cdots b_n^2},  
 $$
 where $q$ and $r$ are polynomials of degree $n-2$. 
 
 Since   $\ce_1$ and $\ce_2$ have equal volumes, the latter equation reduces to $q(u^2) = r(u^2)$, which has at most $n-2$ positive real roots. 
 
 Now choose such a non-zero polynomial $P$ that changes sign exactly at those positive numbers where $F$ changes its sign. This guarantees that the product $PF$ is either everywhere non-negative or non-positive. Since the integral of $PF$ is zero, we conclude that $F$ is identically zero. This means that the sets of numbers $\{a_1, \ldots, a_n\}$ and  $\{b_1, \ldots, b_n\}$ coincide. The theorem is proved.
  
\qed

We will now explain how to adjust the proof above to the case when dual volumes of other orders are given. Let us take any non-zero numbers $i_k$, $k=1,\dots,n$ from the interval $(-2,n]$. Using either formula \eqref{FormulaW_i1} or \eqref{FormulaW_i2}, the equality  $\widetilde  V_{i_k}(\ce_1)=\widetilde V_{i_k}(\ce_2)$ gives 
\begin{equation}\label{eqns}
  \int\limits_0^\infty  {u^{i_k-1}} F(u)\, du =0,
 \end{equation}
if $i_k\ne n$. Here, $F$ is the same function as in \eqref{F}.

If some $i_k$ is equal to $n$ then $F$ has at most $n-2$ positive real roots, and we proceed as above with the help of the following lemma.

\begin{lemma}[Lemma 19, \cite{ENT}] \label{lem:moments}
	Let $f:[0,\infty)\to \mathbb R$ be a continuous function that changes sign at most $N-1$ times in the interval   $(0, \infty)$. If there exist $N$ real numbers $p_1$,..., $p_N$ such that 
	\begin{align*}
  \int\limits_0^{\infty} t^{p_k} f(t) dt =0,\quad \textrm{for every} \quad k=1, \ldots, N,
	\end{align*}
	 then $f$ is identically equal to zero.
\end{lemma}

If neither of $i_k$ is equal to $n$, then we can only conclude that $F$ has at most $n-1$ positive real roots, but we have $n$ conditions of the form \eqref{eqns}, so we can still use the above lemma.

To show how to extend Theorem \ref{thm:main1} to a larger set of indices, let us first give a few formulas that follow from the proof of Lemma \ref{MainFormula}.
For $-4<i<-2$ formula \eqref{frdir} yields \pagebreak
\begin{multline*}	\widetilde{V}_i(\ce) = \frac{4 \pi^{n/2}}{n \Gamma\left(\frac{n-i}2\right)\Gamma\left(\frac{i}{2}\right) }   \\  
	\times  \int\limits_0^\infty u^{i -1} \left( \frac{1}{\sqrt{(1 + \frac{u^2 }{a_1^2})\cdots (1 + \frac{u^2}{a_n^2})}}-1+\frac12\left(\frac{1 }{a_1^2}+\cdots+\frac{1 }{a_n^2}\right)u^2 \right) \, du. 
	\end{multline*}  
When $i=-2$, we have 
$$	\widetilde{V}_{-2}(\ce) = \frac{ \pi^{n/2}}{n \Gamma\left(\frac{n+2}2\right)  }  \left(\frac{1 }{a_1^2}+\cdots+\frac{1 }{a_n^2}\right) .$$

For $n+2<i<n+4$ we get
\begin{multline*}	\widetilde{V}_i(\ce) = \frac{4 \pi^{n/2}}{n \Gamma\left(\frac{n-i}2\right)\Gamma\left(\frac{i}{2}\right) }  \\ \times \int\limits_0^\infty u^{i-1} \left( \frac{1}{\sqrt{(1 +  \frac{u^2 }{a_1^2})\cdots (1 +  \frac{u^2}{a_n^2})}}-\frac{a_1\cdots a_n}{u^n}+\frac{a_1\cdots a_n\left( {a_1^2}+\cdots+ {a_n^2}\right)}{2u^{n+2}}\right) \, du. 
 \end{multline*}
When $i=n+2$, we have 
$$	\widetilde{V}_{n+2}(\ce) = \frac{ \pi^{n/2}  a_1\cdots a_n}{n \Gamma\left(\frac{n+2}2\right)  }  \left( {a_1^2}+\cdots+ {a_n^2}\right) .$$

Now observe that in some cases we can take the interval $(-2,n+2)$ instead of the interval $(-2,n]$. If we assume that $\widetilde  V_n(\ce_1)=\widetilde V_n(\ce_2)$ and $\widetilde  V_{i_k}(\ce_1)=\widetilde V_{i_k}(\ce_2)$ for some distinct numbers $i_k\in (-2,n+2)\setminus\{0,n\}$, $k=1,\ldots,n-1$, then along with the formulas \eqref{FormulaW_i1} and \eqref{FormulaW_i2}, we can also use \eqref{FormulaW_i3}, since $\widetilde  V_i(\ce_1)-\widetilde V_i(\ce_2)=0$ with $n<i<n+2$ gives 
$$	  \int\limits_0^\infty  {u^{i-1}} F(u)\, du =0,$$ with the same function $F$ as used above. 

Using this argument one can proceed even further. For example, if 	$\widetilde{V}_{-2}(\ce_1) = \widetilde{V}_{-2}(\ce_2)$ and $\widetilde{V}_{n}(\ce_1) = \widetilde{V}_{n}(\ce_2)$, then we can take the remaining $n-2$ indices from the interval $(-4,n+2)$. We will have $n-2$ conditions of the form \eqref{eqns}, while the function $F$ will have at most positive $n-3$ roots. To explain the latter, observe that the equation
$$ \left(1 + \frac{u^2}{a_1^2}\right)\cdots \left(1 + \frac{u^2}{a_n^2}\right)=\left(1 +   \frac{u^2}{b_1^2}\right)\cdots \left(1 + \frac{u^2}{b_n^2}\right) $$
  can be written in the form
\begin{multline*}
1+u^2\left( \frac{1}{a_1^2}+\cdots  + \frac{1}{a_n^2}\right)+ u^4 q(u^2) +\frac{u^{2n}}{a_1^2\cdots a_n^2} \\= 1 + u^2\left( \frac{1}{b_1^2}  +\cdots+ \frac{1}{b_n^2}\right)+u^4 r(u^2) + \frac{u^{2n}}{b_1^2\cdots b_n^2},  
\end{multline*}
where $q$ and $r$ are polynomials of degree $n-3$. 

Using similar ideas, one can get uniqueness results involving $\widetilde{V}_{-4}$, $\widetilde{V}_{n+2}$, and so forth.

\section{Intrinsic volumes of ellipsoids in $\mathbb R^3$}

In this section we will give an alternative proof of the following result of Petrov and Tarasov \cite{PT}.

\begin{theorem}
	Let    $\ce_1$ and $\ce_2$	be two ellipsoids in $\R^3$ such that $  V_1(\ce_1)= V_1(\ce_2)$, $  V_2(\ce_1)= V_2(\ce_2)$,  $  V_3(\ce_1)= V_3(\ce_2)$. Then $\ce_1$ and $\ce_2$ are congruent.
\end{theorem}
\begin{proof}
Let  $\ce$  be an ellipsoid in $\mathbb R^3$. Using formulas \eqref{1stIV}, \eqref{tildeVi}, and \eqref{relation_dual}, we can write its first intrinsic volume as follows:
$$
V_1(\ce ) = \frac{1}{\pi} \int \limits_{S^{2}} h_{\ce }(\theta) d\theta = \frac{1}{\pi} \int \limits_{S^{2}} \|\theta\|_{\ce ^\circ}  d\theta = \frac{3}{\pi}  \widetilde V_{-1}(\ce ^\circ)  = \frac{4}{V_3(\ce)} \widetilde V_{4}(\ce) .
$$
To find the second intrinsic volume of $\ce$, we will use formulas \eqref{KZ}, \eqref{1stIV}, and \eqref{tildeVi} .
\begin{align*}
V_2(\ce)& = \frac{3}{8} V_3(\ce) \cdot V_1(\ce^\circ) =    \frac{3}{8\pi}V_3(\ce) \int_{S^2} h_{\ce^\circ} (\theta) \, d\theta\\
&=  \frac{3}{8\pi}V_3(\ce) \int_{S^2} \|\theta\|_{\ce}   \, d\theta=  \frac{9}{8\pi} V_3(\ce)  \widetilde V_{-1}(\ce). 
\end{align*}

Thus the knowledge of $V_1$, $V_2$, $V_3$ is equivalent to the knowledge of $ \widetilde V_{-1}$, $ \widetilde V_4$, $ \widetilde V_3$. The latter information determines the ellipsoid uniquely, as  was shown in the previous section. 
\end{proof} 

\section{Intrinsic volumes of ellipsoids of revolution in $\R^n$}

In this section we will discuss uniqueness of ellipsoids of revolution with given intrinsic volumes. As we will see below, we do not need to know all intrinsic volumes $V_1, \ldots, V_n$, only a few of them.

Let $K$ be an origin-symmetric convex body. $K$ is called a zonoid if $K$ can be approximated in the Hausdorff metric by finite Minkowski sums of segments.  It is known (see \cite[Thm 3.5.3]{Sch}) that $K$ is a zonoid if and only if 
$$
h_{K}(x) = \int \limits_{S^{n-1}} |\langle x , \theta\rangle | d \mu(\theta), \quad x \in \mathbb R^{n}, 
$$
where $\mu$ is a  non-negative even Borel measure on $S^{n-1}$.   The measure $\mu$ is called the generating measure of $K$. 

Let $\ce$ be an ellipsoid in $\mathbb R^n$ centered at the origin. It is easy to see that $\ce$ is a zonoid. Moreover, its generating measure is given by
$$d\mu(\theta) =\frac{\kappa_n}{2 \mathrm{vol}(\ce) \kappa_{n-1}} \|\theta\|_{\ce}^{-n-1} d\theta.$$
Indeed, let $A$ be a  linear transformation such that $\ce = A B_2^n$ . Then
\begin{align*} \int \limits_{S^{n-1}} |\langle x , \theta\rangle |\, \|\theta\|_{\ce}^{-n-1} d\theta & = (n+1) \int \limits_{\ce} |\langle x , y\rangle |\, dy = (n+1) |\det A|  \int \limits_{B_2^n} |\langle x , A y\rangle |\, dy
\\
&= (n+1) |\det A|  \int \limits_{B_2^n} |\langle A^T x ,  y\rangle |\, dy \\
& = (n+1) |\det A| \,|A^T x| \int \limits_{B_2^n} \left|\left\langle \frac{A^T x}{|A^T x|} ,  y\right\rangle \right|\, dy \\
& = (n+1) |\det A| \, |A^T x| C_n,
\end{align*} 
where $$C_n=  \int \limits_{B_2^n} \left|\left\langle u ,  y\right\rangle \right|\, dy = \frac{1}{n+1} \int \limits_{S^{n-1}} \left|\left\langle u ,  \theta \right\rangle \right|\, d\theta = \frac{2 \kappa_{n-1}}{n+1} ,$$
for any $u\in S^{n-1}$; see e.g., \cite[Lemma 3.4.5]{Gr}.

Thus, 
$$\int \limits_{S^{n-1}} |\langle x , \theta\rangle |\, \|\theta\|_{\ce}^{-n-1} d\theta  = 2 \kappa_{n-1} |\det A|\, |A^T x|  = 2 \frac{\kappa_{n-1}}{\kappa_n} \mathrm{vol}(\ce)  h_{A B_2^n}(x)     =  2  \frac{\kappa_{n-1}}{\kappa_n} \mathrm{vol}(\ce) h_{\ce }(x) ,
$$
as claimed.

Recall (see formulas (5.31) and (5.82) in \cite{Sch}) that   the $k$-th intrinsic volume of a zonoid $K$ with generating measure $\mu$ is given by
\begin{equation}\label{fml:zonoids}
V_k(K) = C(n,k) \int \limits_{S^{n-1}} \ldots \int \limits_{S^{n-1}} D_n(v_1, \ldots, v_n)\, d\mu(v_1)  \cdots d \mu(v_k) \, dv_{k+1} \cdots dv_n,
\end{equation}
where $D_n(v_1, \ldots,v_n)$ is the  volume of the parallelepiped spanned by the vectors $v_1, \ldots, v_n$, and $C(n,k)= \displaystyle\frac{2^k {n\choose k}}{n!\, \kappa_{n-k} \kappa_{n-1}^{n-k}}$.

Therefore, for an ellipsoid $\ce = AB_2^n$ we have
\begin{align*} 
V_k(\ce) &= \frac{\kappa^k_n}{2^k \mathrm{vol}(\ce)^k \kappa_{n-1}^k}C(n,k)  \int \limits_{S^{n-1}} \ldots \int \limits_{S^{n-1}} D_n(v_1, \ldots, v_n)\, \|v_1\|_{\ce}^{-n-1} d v_1   \cdots \|v_k\|_{\ce}^{-n-1} d  v_k  \, dv_{k+1} \cdots dv_n\\
 &= \frac{C(n,k)(n+1)^n \kappa^k_n}{2^k \mathrm{vol}(\ce)^k \kappa_{n-1}^k}  \underbrace{\int \limits_{\ce} \ldots \int \limits_{\ce} }_{k} \underbrace{\int \limits_{B_2^n} \ldots \int \limits_{B_2^n}}_{n-k} D_n(v_1, \ldots, v_n)\,   d v_1   \cdots dv_n\\
  &=\frac{C(n,k)(n+1)^n }{2^k  \kappa_{n-1}^k}  \int \limits_{B_2^n} \ldots \int \limits_{B_2^n}  D_n(A v_1, \ldots, A v_k,   v_{k+1}, \ldots  v_n)\,   d v_1   \cdots dv_n\\
   &= \frac{C(n,k) }{2^k  \kappa_{n-1}^k}  \int \limits_{S^{n-1}} \ldots \int \limits_{S^{n-1}}  D_n(A v_1, \ldots, A v_k,  v_{k+1}, \ldots  v_n)\,   dv_1   \cdots dv_n.
\end{align*}

Now assume that $\ce$ is an ellipsoid of revolution of the form
$$\ce=\left\{x\in \mathbb R^n: \frac{x_1^2}{a^2} + \cdots \frac{x_{n-1}^2}{a^2}+ \frac{x_n^2}{b^2}\le 1\right\}.$$
The corresponding matrix $A$ is diagonal, with the entries  $a, \ldots, a, b$ on the main diagonal. It will be convenient to use the following notation: $A = \mathrm{diag} \, \{a, \ldots, a, b\}$.
We have
\begin{align*} D_n(A v_1, \ldots, A v_k,   v_{k+1} \ldots  v_n)& = \mathrm{abs} \left| \begin{array}{llllll}  a v_{1}^{1} &  \ldots & a v_{k}^{1} & v_{k+1}^{1}
& \ldots &  v_{n}^{1}\\
 a v_{1}^{2} &  \ldots & a v_{k}^{2} & v_{k+1}^{2}
& \ldots &   v_{n}^{2}\\
\vdots  & \ldots & \vdots  &\vdots 
& \ldots & \vdots \\
 a v_{1}^{n-1}& \ldots & a v_{k}^{n-1} & v_{k+1}^{n-1}
& \ldots &   v_{n}^{n-1}\\
 b v_{1}^{n}& \ldots & b v_{k}^{n} & v_{k+1}^{n}
& \ldots &  v_{n}^{n}
\end{array} \right| \\
& = a^ k \mathrm{abs} \left| \begin{array}{llllll}    v_{1}^{1} &  \ldots &   v_{k}^{1} & v_{k+1}^{1}
& \ldots &  v_{n}^{1}\\
 v_{1}^{2} &  \ldots &  v_{k}^{2} & v_{k+1}^{2}
& \ldots &   v_{n}^{2}\\
\vdots  & \ldots & \vdots  &\vdots 
& \ldots & \vdots \\
  v_{1}^{n-1}& \ldots &   v_{k}^{n-1} & v_{k+1}^{n-1}
& \ldots &   v_{n}^{n-1}\\
\frac{b}{a} v_{1}^{n}& \ldots & \frac{b}{a} v_{k}^{n} & v_{k+1}^{n}
& \ldots &  v_{n}^{n}
\end{array} \right| ,
\end{align*}
where abs stands for the absolute value and $v_i^j$ denotes the $j$-th coordinate of vector $v_i$.

Let $Q=\mathrm{diag}\Big\{  \underbrace{\frac{b}{a},\ldots,\frac{b}{a}}_{k}, \underbrace{1, \ldots,1}_{n-k} \Big\}$. 
Denoting $v^i= (v_1^i, \ldots v_n^i)$, we see that 
$$ D_n(A v_1, \ldots, A v_k,   v_{k+1}, \ldots  v_n) = a^k  D_n( v^1, \ldots,   v^{n-1},   Q  v^n).$$

Therefore, using an integral representation similar to \eqref{repr}, we get
\begin{align*} V_k(\ce)& =  \frac{C(n,k) }{2^k  \kappa_{n-1}^k}  \int \limits_{S^{n-1}} \ldots \int \limits_{S^{n-1}}  D_n(A v_1, \ldots, A v_k,   v_{k+1}, \ldots  v_n)\,   d v_1   \cdots dv_n\\
& = \frac{2^{n-k} C(n,k) }{ \kappa_{n-1}^k \left(\Gamma\left(\frac{n+1}2\right)\right)^n}  \int \limits_{\mathbb R^n} \ldots \int \limits_{\mathbb R^n}  D_n(A v_1, \ldots, A v_k,   v_{k+1}, \ldots  v_n)e^{-|v_1|^2 - \cdots - |v_n|^2}\,   d v_1   \cdots dv_n\\
& = \frac{2^{n-k}   C(n,k) a^k}{ \kappa_{n-1}^k \left(\Gamma\left(\frac{n+1}2\right)\right)^n}  \int \limits_{\mathbb R^n} \ldots \int \limits_{\mathbb R^n}   D_n( v^1, \ldots,   v^{n-1},   Q  v^n) e^{-|v^1|^2 - \cdots - |v^n|^2}\,   d v^1   \cdots dv^n\\
& = \frac{2^{-k} C(n,k)  a^k}{  \kappa_{n-1}^k }  \int \limits_{S^{n-1}} \ldots \int \limits_{S^{n-1}}  D_n( v^1, \ldots,  v^{n-1},   Q  v^n) \,   d v^1   \cdots dv^n\\
& = \displaystyle\frac{{n\choose k} \kappa_{n-1}}{n \, \kappa_{n-k} } a^k\,  V_1(\tilde \ce),
\end{align*} 
where $\tilde \ce = QB_2^n$. 

Thus, by \eqref{1stIV},
$$V_k(\ce) = \bar C(n,k)  a^k\, \int_{S^{n-1}} \left(\frac{b^2}{a^2} (\theta_1^2+\cdots + \theta_k^2)+ \theta_{k+1}^2+\cdots + \theta_n^2 \right)^{1/2} \, d\theta, $$
where $\bar C(n,k) =  \displaystyle \frac{{n\choose k}}{n\kappa_{n-k}}$.

Now suppose that we are given $V_n(\ce)$ and $V_k(\ce)$, for some $k$, $1\le k\le n-1$. It is easy to see that these two intrinsic volumes do not determine an ellipsoid of revolution uniquely. Indeed, after a rescaling, we can assume that $\ce$ has the same volume as the ball $B_2^n$, i.e., $a^{n-1} b=1$. Thus,
\begin{equation}\label{V_kREV}
V_k(\ce) = \bar C(n,k) \int_{S^{n-1}} \Big( a^{2k-2n} (\theta_1^2+\cdots + \theta_k^2)+ a^{2k}(\theta_{k+1}^2+\cdots + \theta_n^2) \Big)^{1/2} \, d\theta. 
\end{equation} 

The latter integral is a continuous function of $a\in (0,\infty)$, and it approaches infinity as $a\to 0$ and $a\to \infty$. This shows that there are two non-congruent ellipsoids of revolution with equal $V_n$ and $V_k$. Thus in order to prove a result in the positive direction, we need more intrinsic volumes. 

\begin{theorem}
Let $\ce_1$ and $\ce_2$ be two ellipsoids of revolution in $\mathbb R^n$ such that $V_n(\ce_1) = V_n(\ce_2)$,  $V_k(\ce_1) = V_k(\ce_2)$, $V_{n-k}(\ce_1) = V_{n-k}(\ce_2)$ for some $k\ne n/2$, $1\le k\le n-1$. Then $\ce_1$ and $\ce_2$ are congruent. 

If $n$ is even, then $V_n$, $V_{n/2}$, $V_k$, for any $k$ different from $n$ and $n/2$, uniquely determine an ellipsoid of revolution (up to an isometry). 
\end{theorem}
\begin{proof}
	Without loss of generality we can assume that the volumes of $\ce_1$ and $\ce_2$ are equal to the volume of $B_2^n$. Thus we will deal with two ellipsoids whose semiaxes are $a,\ldots,a, a^{-n+1}$ and $b,\ldots,b, b^{-n+1}$ correspondingly. Our goal is to show that $a=b$.

	Consider the ellipsoids $\tilde \ce_1 = Q_a B_2^n$ and  $\tilde \ce_2 = Q_b B_2^n$, where 
	$$ Q_a=\mathrm{diag}\,\{\underbrace{a^{-2k+2n},\ldots,a^{-2k+2n}}_{k}, \underbrace{a^{-2k}, \ldots,a^{2k}}_{n-k}\} $$ and $$ Q_b=\mathrm{diag}\,\{\underbrace{b^{-2k+2n},\ldots,b^{-2k+2n}}_{k}, \underbrace{b^{-2k}, \ldots,b^{-2k}}_{n-k}\}. $$
	By formula \eqref{V_kREV}, we see that 
	 $V_k(\ce_1) = V_k(\ce_2)$ is equivalent to 	$\widetilde V_{-1}(\tilde \ce_1) = \widetilde V_{-1}(\tilde \ce_2)$. Replacing $k$ by $n-k$ in \eqref{V_kREV}, we see that $V_{n-k}(\ce_1) = V_{n-k}(\ce_2)$ is equivalent to $\widetilde V_{-1}(\tilde \ce_1^\circ) = \widetilde V_{-1}(\tilde \ce_2^\circ)$, which in turn is equivalent to $\widetilde V_{n+1}(\tilde \ce_1) = \widetilde V_{n+1}(\tilde \ce_2)$, by formula \eqref{relation_dual}. Thus Lemma \ref{MainFormula} gives 
	 $$	  \int\limits_0^\infty  {u^{-2}} F(u)\, du =0  \quad\mbox{and} \quad  \int\limits_0^\infty  {u^{n}} F(u)\, du =0,$$
	 where 
	 $$F(u)= \frac{1}{(1+a^{2k-2n}u^2)^{\frac{k}{2}}(1+a^{2k}u^2)^{\frac{n-k}{2}}} - \frac{1}{(1+b^{2k-2n}u^2)^{\frac{k}{2}}(1+b^{2k}u^2)^{\frac{n-k}{2}}}.$$
	 To use the ideas of Section \ref{SectionMR} it remains to show that $F$ has at most one positive root, unless it is identically equal to zero.
	 To this end observe that $F$ has the same roots as the function $g$ given by
	 $$g(u) = \ln \left((1+a^{2k-2n}u^2)^{\frac{k}{2}}(1+a^{2k}u^2)^{\frac{n-k}{2}}\right) - \ln \left( (1+b^{2k-2n}u^2)^{\frac{k}{2}}(1+b^{2k}u^2)^{\frac{n-k}{2}}\right).$$
Computing its derivative, we get
	 \begin{align*} g'(u)  = 
	 u\left(\frac{ka^{2k-2n}-kb^{2k-2n}}{(1+a^{2k-2n}u^2)(1+b^{2k-2n}u^2)}+\frac{(n-k)a^{2k}-(n-k)b^{2k}}{(1+a^{2k}u^2)(1+b^{2k}u^2)}   \right).
	 \end{align*}
	 One can see that $g'$ has at most two positive real roots. Since $g(0)=0$ and $\lim_{u\to \infty}g(u) = 0$, we conclude that $g$ has at most one positive root.
	  
	Now assume $n$ is even and $V_n(\ce_1) = V_n(\ce_2)=V_n(B_2^n)$, $V_{n/2}(\ce_1) = V_{n/2}(\ce_2)$, $V_k(\ce_1) = V_k(\ce_2)$ for some $k\ne n/2, n$. By Lemma \ref{MainFormula}, $V_{n/2}(\ce_1)$ is (up to a multiplicative constant) equal to 
	$$\int\limits_0^\infty  {u^{-2}} \left(1 - \frac{1}{\big((1+a^{-n}u^2) (1+a^{n}u^2)\big)^{\frac{n}{4}}}\right)\, du. $$
	 The function $(1+a^{-n}u^2) (1+a^{n}u^2)$ considered as a function of $a$ is decreasing from $0$ to $1$ and increasing from $1$ to $\infty$ (for every $u$). Thus the integral above is decreasing/increasing on the same intervals with respect to $a$. Moreover, the integral is invariant under the transformation $a\to a^{-1}$. 
	 Thus, if $V_{n/2}$ is prescribed a given value, then there are at most  two ellipsoids of revolution with this intrinsic volume and they must be polars of each other. Hence either $\ce_2=\ce_1$ or $\ce_2=\ce_1^\circ$. Assume the latter. By virtue of formula \eqref{KZ}, $V_k(\ce_1) = V_k(\ce_1^\circ)$ yields $V_{n-k}(\ce_1) = V_{n-k}(\ce_1^\circ)$. By what we proved above, this information guarantees that $\ce_1=\ce_1^\circ$. Thus $\ce_1=B_2^n$, and therefore $\ce_2=B_2^n$. So $\ce_1$ and $\ce_2$ cannot be different.
	\end{proof}


\begin{thebibliography}{4}
	
	\bibitem{ENT} {\sc A.~Eskenazis, P.~Nayar, T.~Tkocz}, {\it Sharp comparison of moments and the log-concave moment problem}, Adv. Math. {\bf 334} (2018) 389--416.

	\bibitem{G} {\sc R.~J.~Gardner}, {\it Geometric tomography, 2nd edition}., Encyclopedia  Math.  Appl., 58,  Cambridge University Press, 2006.
	
	\bibitem{GS} {\sc I.~M.~Gelfand, G.~E.~Shilov}, {\it Generalized functions, vol.1 Properties and Operations}, Academic Press, New York and London, 1964.
	
	
	\bibitem{Gr} {\sc H.~Groemer}, {\em Geometric Applications of Fourier Series and Spherical Harmonics},  Cambridge University Press, New York, 1996.
	
 
	
	\bibitem{KZ}  {\sc Z.~Kabluchko, D.~Zaporozhets}, {\it  Intrinsic volumes of Sobolev balls with application to Brownian convex hulls}, Trans. Amer. Math. Soc. {\bf 368} (2016), 8873--8899.
	
	\bibitem{K}
	{\sc A.~Koldobsky}, {\it Fourier Analysis in Convex Geometry}, American Mathematical Society, Providence RI, 2005.
	
 
	
	\bibitem{L} {\sc E.~Lutwak}, {\it Dual mixed volumes}, Pacific J. Math. {\bf 58} (1975), 531--538.
	
	\bibitem{PT} {\sc F.~Petrov, A.~Tarasov}, {\it Uniqueness of a three-dimensional ellipsoid with given intrinsic volumes},  Arnold Math J. (2020);  arXiv:1905.01728v1.
	
 
	
\bibitem{Sch}  {\sc R.~Schneider,} {\em Convex bodies:
	the Brunn-Minkowski theory}, Second expanded edition. Encyclopedia  Math.  Appl., 151. Cambridge University Press, Cambridge, 2014. 
\end{thebibliography}
\end{document}